\documentclass{amsart}
\usepackage[utf8]{inputenc}
\usepackage[T1]{fontenc}
\usepackage[english]{babel}
\usepackage{color}
\usepackage{amsmath} 
\usepackage{amssymb} 
\usepackage{amsthm}
\usepackage{graphicx}
\usepackage[colorlinks=true, linkcolor=blue]{hyperref}
\usepackage{cancel}
\usepackage{lmodern}
\usepackage{microtype}
\usepackage{graphicx}

\theoremstyle{plain}
\newtheorem{thm}{Theorem}
\newtheorem*{thm*}{Theorem}
\newtheorem{prop}[thm]{Proposition} 
\newtheorem{lem}[thm]{Lemma} 

\newtheorem*{cor*}{Corollary}

\newtheorem{defi}[thm]{Definition}

\newtheorem{rem}[thm]{Remark}

\newcommand {\R} {\mathbb{R}} 
\newcommand {\T} {\mathbb{T}} \newcommand {\N} {\mathbb{N}}

\newcommand {\p} {\partial}
\newcommand {\dt} {\partial_t}

\begin{document}
\title[Stability in Gevrey Spaces]{Linear Inviscid Damping in Sobolev and Gevrey
Spaces} \author{Christian Zillinger}

\begin{abstract}
  In a recent article \cite{jia2019linear} Jia established linear inviscid
  damping in Gevrey regularity for compactly supported Gevrey regular shear
  flows in a finite channel, which is of great interest in view of existing
  nonlinear results \cite{deng2018}, \cite{bedrossian2013asymptotic},
  \cite{ionescu2018inviscid}. In this article we provide an alternative very
  short proof of stability in Gevrey regularity as a consequence of stability in
  high Sobolev regularity \cite{Zill3}, \cite{Zill5}. Here, we consider both the setting
  of a finite channel with compactly supported perturbations and of an infinite
  channel without this restriction. Furthermore, we consider the setting where
  perturbations vanish only of finite order.
\end{abstract}

\maketitle

\noindent
In recent years the asymptotic stability of the Euler equations
\begin{align*}
  \dt v + v \cdot \nabla v + \nabla p =0,
\end{align*}
near shear flow solutions $v=(U(y),0)$ has been an area of very active research.
Following the works of Mouhot and Villani \cite{Villani_short} on Landau damping
in plasma physics, in a seminal work Bedrossian and Masmoudi
\cite{bedrossian2013asymptotic} for the first time established nonlinear
asymptotic stability and damping for the prototypical case $U(y)=y$, known as
Couette flow. Here Gevrey regularity plays a crucial role in controlling
nonlinear resonances, so called \emph{echoes} \cite{deng2018}, \cite{dengZ2019},
\cite{bedrossian2013landau}. In contrast, the linear problem is known to stable
in (arbitrary) Sobolev regularity \cite{Zill3}, \cite{grenier2019linear}, \cite{Zhang2015inviscid} for the setting without boundary,
but only stable in (optimal) low Sobolev regularity for the setting with
boundary \cite{Zill5}, \cite{Zhang2015inviscid} unless shear perturbation
vanishes on the boundary. In a recent work \cite{jia2019linear} Jia thus studied
the problem of linear asymptotic stability of compactly supported perturbations
to Couette flow in Gevrey regularity.

As the main results of this article we show that:
\begin{itemize}
\item Stability in Gevrey regularity corresponds to a quantitative control of
  the stability in Sobolev spaces. In particular, we show that the control
  established in \cite{Zill3} yields a very short proof of stability in Gevrey
  classes for the setting of an infinite channel. Furthermore, only a
  quantitative stability result in $L^2$ is needed, which then implies all
  higher stability results.
\item In \cite{jia2019linear} Jia considers the question of stability in Gevrey
  regularity for the setting of a finite channel, where the shear flow and the
  vorticity perturbation are compactly supported away from the boundary. In view
  of the boundary instabilities established in \cite{Zill5} such a restriction
  might be necessary. In this work we show that under such a support condition
  stability in Gevrey regularity for the setting of a finite channel essentially
  reduces to the setting without boundary with minor correction terms similar to the
  $H^1$ stability problem considered in \cite{Zill5}.
\item As a further result, we establish stability in Sobolev spaces $H^j$,
  $j\leq N$ in the setting of a finite channel if $U''$ and $\omega_0$ vanish up to order $N$
  on the boundary.
\item In this article we restrict ourselves to considering small, smooth
  bilipschitz shear flows and circular flows close to (Taylor-)Couette flow. We
  expect a further extension to more general and degenerate shear flows and
  circular flows in weighted spaces in an analogous way to
  \cite{coti2019degenerate} to be possible with some technical effort.
\end{itemize}
In this sense the core problem of linear inviscid damping lies in establishing
$L^2$ stability and $H^1$ stability (which has to account for some boundary
effects). The setting of higher regularity then follows by an iteration scheme.

We remark that our theorems impose a smallness condition which is sufficient but
not necessary. As shown in \cite{Zhang2015inviscid} a more precise condition is
given by requiring that there are no embedding eigenvalues of the associated
Rayleigh problem. Our stronger condition allows us to construct a
Lyapunov functional using perturbative methods.
\\

The linearized Euler equations around a shear flow $U(y)$ are given by
\begin{align*}
  \dt \omega + U(y)\p_x + U''(y)\p_x \Delta^{-1} \omega=0.
\end{align*}
Here, $\omega_0$ and hence $\omega$ is understood to without loss of generality
have zero mean in $x$ and in the setting of a finite channel $\T_L \times
[0,1]$, $\Delta^{-1}$ imposes zero Dirichlet boundary conditions in $y$.
Assuming that $U(y)$ is Bilipschitz, we change variables by $y=U^{-1}(z)$ and
denote
\begin{align*}
  f(z)=U''(U^{-1}(z)), \\
  g(z)= U'(U^{-1}(z)),
\end{align*}
and further pass to Lagrangian coordinates $(t,x+tz,z)$. With respect to these
coordinates our problem is given by
\begin{align}
  \label{eq:linEuler1}
  \begin{split}
    \dt \omega + f \p_x L_t \omega &=0, \\
    L_t&= (\p_x^2 + (g(\p_z-t\p_x))^2)^{-1},
  \end{split}
\end{align}
where again in the setting of a finite channel $L_t$ satisfies zero Dirichlet
boundary conditions.

We recall that Gevrey classes measure the growth of $\mathcal{C}^{j}$ or $H^{j}$
norms as $j\rightarrow \infty$. See \cite{ionescu2018inviscid},
\cite{GevreyEncy} and \cite[page 281]{hormander2015analysis}.
\begin{defi}
  \label{defi:Gevrey}
  Let $f \in C^{\infty}$ and $s \in [1,\infty)$. We then introduce the following
  three related but distinct definitions of the Gevrey class $\mathcal{G}_s$:
  \begin{enumerate}
  \item We say $f$ is in the $L^\infty$ based Gevrey class
    $\mathcal{G}_{s}^\infty$ if there exists a constant $C>0$ such that
    \begin{align*}
      \|f\|_{\mathcal{C}^{j}} \leq C^{1+j} (1+j)^{s j}
    \end{align*}
    for all $j \in \N$.
  \item We say that $f$ is in the (Sobolev based) Gevrey class $\mathcal{G}_{s}$
    if there exists a constant $C>0$ such that
    \begin{align*}
      \|f\|_{H^{j}} \leq C^{1+j} (1+j)^{s j}
    \end{align*}
    for any $j \in \N$
  \item We say that $f$ is in the second Sobolev bases Gevrey class if there
    exists a constant $\lambda>0$ such that
    \begin{align*}
      \int \exp(\lambda \langle \xi \rangle^{\frac{1}{s}}) |\tilde{f}(\xi)|^2 d\xi < \infty. 
    \end{align*}
  \end{enumerate}
\end{defi}

\begin{rem}
  \begin{itemize}
  \item We remark that in the literature also a parametrization in terms of
    $\frac{1}{s} \in (0,1]$ is common.
  \item A more general version of the first definition considers the restriction
    of $f$ to compact sets $K$ and constants $C_K$. For example, $f(x)=x$ does
    not satisfy our definition, since we impose that $f(x)$ is bounded
    uniformly. However, since our theorems impose these constraints on
    derivatives of $U(y)$, $U(y)$ itself may be close to affine.
  \item We may use a Sobolev embedding to estimate $\|f\|_{C^j}\leq
    \|f\|_{H^{j+N}}$, where $N>0$ depends on the dimension. Increasing $s$
    slightly and increasing $C$, we thus see that every $f$ in the first Sobolev
    based Gevrey class is also contained in the $L^\infty$ based Gevrey class.
  \item Expressing $\exp(\lambda \langle \xi \rangle^{\frac{1}{s}})$ as a
    series, the last definition implies that
    \begin{align*}
      \lambda^{j} \frac{1}{j!} \|f\|_{H^{\frac{j}{2s}}}^2 \leq C
    \end{align*}
    for all $j \in \N$. Expressing the factorial using the Stirling
    approximation and considering $\sigma=\frac{j}{2s}$, we thus see that any
    such function also satisfies the second definition.
  \item Conversely, we may round up $\frac{j}{2s}$ in the series expansion to
    show that the second definition also implies the third with an arbitrarily
    small loss in $s$.
  \item In our analysis we consider the regularity of coefficient functions
    according to the first definition and the regularity of the vorticity
    according to the second definition.
  \end{itemize}
\end{rem}

The following three theorems summarize our main results. 
We first consider the case of an infinite channel $\T_{L} \times \R$,
for which we had previously established non-quantitative stability results in
\cite{Zill3}.
The following theorem improves this to quantitative estimates for each $H^j$ and
thus to Gevrey regularity.
Here, as we will see in Section \ref{sec:infinite}, the core of the proof is
given by establishing quantitative stability in $L^2$, from which higher
regularity follows by a short inductive argument.

\begin{thm}[Summary infinite channel]
  \label{thm:infinite}
  Consider the linearized Euler equations \eqref{eq:linEuler1} on $\T_{L} \times
  \R$ around a bilipschitz shear flow $(U(y),0)$. There exists $c>0$ such that
  if
  \begin{align}
    \label{eq:smallness}
    \|f\|_{W^{1,\infty}} L<c,
  \end{align}
  then for any $s \in [1,\infty)$ if $f,g \in \mathcal{G}_{s}^{\infty}$ and
  $\omega_0\in \mathcal{G}_s$, the problem \eqref{eq:linEuler1} is stable in Gevrey
  regularity.
  That is, if for all $j \in \N$ it holds that
  \begin{align*}
    \|f\|_{\mathcal{C}^j} + \|g\|_{\mathcal{C}^j} &\leq D_1^{1+j} (1+j)^{j s}, \\
    \|\omega_0\|_{H^{j}}&\leq D_2^{1+j} (1+j)^{j s},
  \end{align*}
  then there exists $C=C(D_1,D_2,c)$ such that for all times $t\geq 0$
  \begin{align}
    \label{eq:16}
    \|\omega(t)\|_{H^{j}}\leq C^{1+j} (1+j)^{j s}.
  \end{align}
\end{thm}
The smallness condition \eqref{eq:smallness} here is imposed in order to allow a
perturbative construction in our stability proof for $L^2$ and sufficient but
not necessary (see Section \ref{sec:infinite} for further discussion).

While the setting of a finite channel $\T_L \times [0,1]$ is in general unstable
in higher Sobolev regularity \cite{Zill5}, it turns out that in the setting of
compactly supported perturbations studied in \cite{jia2019linear} all boundary
effects can be easily controlled and an analogous stability result hold.

\begin{thm}[Summary finite channel]
  \label{thm:finite}
  Consider the linearized Euler equations \eqref{eq:linEuler1} on $\T_{L} \times
  [0,1]$ around a bilipschitz shear flow flow $(U(y),0)$, with $U'\geq 1$. There
  exists $c>0$ such that if
  \begin{align}
    \|f\|_{W^{1,\infty}} L<c,
  \end{align}
  then for any $s \in [1,\infty)$ if $f,g \in \mathcal{G}_{s}^{\infty}$ and
  $\omega_0\in \mathcal{G}_s$ are compactly supported away from the boundary,
  then problem \eqref{eq:linEuler1} on $\T_{L} \times [0,1]$ is stable in Gevrey
  regularity in the sense of Theorem \ref{thm:infinite}.
\end{thm}

Here, the core of the problem lies in establishing stability in $H^{1}$ as in
\cite{Zill5}, from which the desired quantitative higher Sobolev and Gevrey
regularity results then follow by induction.

Finally, we note that it is not necessary to impose the condition of compact
support, but that a high order of vanishing on the boundary is sufficient to
establish stability in Sobolev regularity (or Gevrey regularity).
\begin{thm}[Finite regularity for a finite channel]
  \label{thm:finite2}
  Let $g$ satisfy the same assumptions as in the previous theorem and suppose
  that there exists $N \in \N_0$ such that $f$ and the initial vorticity
  perturbation $\omega_0$ vanish to order $N$ on the boundary.
  Suppose further that
  \begin{align*}
    \|f\|_{\mathcal{C}^j} + \|g\|_{\mathcal{C}^j} &\leq D_1^{1+j} (1+j)^{j s}
  \end{align*}
  for all $j \leq N$.
  Then there exists $C=C(D_1,L)$ such that for all $j \leq N$ and all $t\geq 0$
  it holds that
  \begin{align*}
    \|\omega(t)\|_{H^{j}} \leq  C^{1+j}  \|\omega_0\|_{H^{j}}.
  \end{align*}
\end{thm}
In particular, the Gevrey stability result of Theorem \ref{thm:finite} is still valid
if we only assume that $f$ and $\omega_0$ vanish of infinite order.

\section{The Infinite Channel Case}
\label{sec:infinite}

As a starting point we consider the problem \eqref{eq:linEuler1} in the infinite
channel $\T_L\times \R$ and establish the following quantitative improvement of Theorem
4.5 in \cite{Zill3}:
\begin{thm}
  \label{thm:quantitative_stability}
  Suppose that $U\in C^{2}(\R)$ is bilipschitz and let $g(z)=U'(U^{-1}(z))$,
  $f(z)=U''(U^{-1}(z))$. Consider the linearized Euler equations around $U$ on
  $\T_{L}\times \R$ and suppose that
  \begin{align*}
    \|f\|_{W^{1,\infty}} L \ll 1.
  \end{align*}
  Suppose further that for some $j \in \N$, $f,g \in W^{1+j, \infty}$. Then the
  solution $\omega$ (in Lagrangian coordinates) satisfies
  \begin{align*}
    \|\omega(t)\|_{\dot{H}^j}^2\leq  C\|\omega_0\|_{\dot{H}^j}^2+ C 2^j \sum_{j_1+j_2\leq j}  \|\omega_0\|_{\dot{H}^{j_1}}^2 \|(f,g)\|_{j_2},
  \end{align*}
  for all times $t\geq 0$. Here, we used the short notations:
  \begin{align*}
    \|f\|_{j}&:= \sup_{j_1+j_2+\dots+ j_N=j} \prod \|\p_z^{j_i}f\|_{L^\infty}, \\
    \|(f,g)\|_{j}&:=\sum_{j_1+j_2=j} (1+\|f\|_{j_1})(1+\|g\|_{j_2}). 
  \end{align*}
\end{thm}
We note that by definition of $\|\cdot\|_{j}$ it holds that
\begin{align*}
  \|f\|_{j_1} \|f\|_{j_2}\leq \|f\|_{j_1+j_2},
\end{align*}
which helps to simplify recursive (commutator) estimates of the form
\begin{align*}
  a_0&=1,\\
  a_{j+1}&\leq \sum_{j_1+j_2=j} \|f\|_{j_1} a_{j_2}.
\end{align*}
We further remark that if $f \in \mathcal{G}_{s}^\infty$, then $\|f\|_{j}$
satisfies analogous estimates to $\|f\|_{\mathcal{C}^j}$:
\begin{align*}
  \|f\|_{j} &= 2^{j} \sup_{j_1+j_2+\dots+ j_N=j} \prod \|\p_z^{j_i}f\|_{L^\infty} \\
  &\leq 2^{j} \sup \prod C^{1+j_i} (1+j_i)^{j_i} \\
  &\leq 2^{j} C^{2j} (1+j)^{j}.
\end{align*}

In \cite{Zill3} we subsumed the precise bound into a control by $C_j
\|\omega_0\|_{H^j}^2$ for a non-explicit constant $C_j$ and we imposed the stronger constraint that
$\|f\|_{W^{j+1,\infty}} L \ll 1$. However, as already noted and proven in
\cite{Zill5}, \cite{Zill6} only smallness in $W^{1,\infty}$ is actually used in
the proof.

This quantitative control of constants then immediately allows us to establish
the stability in Gevrey classes expressed in Theorem \ref{thm:infinite}.

\begin{proof}[Proof of Theorem \ref{thm:infinite}]
  Let $s \in [1,\infty)$ be given and $U'' \in \mathcal{G}_s^{\infty}, \omega_0
  \in \mathcal{G}_s$. There thus exists constants $D_{1},D_2$ such that
  \begin{align*}
    \|(f,g)\|_{j}\leq D_1^{1+j} (1+j)^{j s}, \\
    \|\omega_0\|_{H^{j}}\leq D_2^{1+j} (1+j)^{j s},
  \end{align*}
  for any $j \in \N$. Applying Theorem \ref{thm:quantitative_stability} we hence
  obtain that
  \begin{align*}
    \|\omega(t)\|_{H^j}^2 &\leq C(\|\omega_0\|_{H^j}^2 +2^{j} \sum_{j_1+j_2=m}C_{j_1} \|\omega_0\|_{H^{j_2}}^2 )\\
                          &\leq C D_1^{2(1+j)} + C 2^{j}\sum_{j_1+j_2=j} C(1+C^{1+j_1}D_2^{2(1+j_1)}(1+j_1)^{2j_1 s}) D_2^{2(1+j_2)} (1+j_2)^{2j_2 s}.
  \end{align*}
  We now note that
  \begin{align*}
    (1+j_1)^{2j_1 s} (1+j_2)^{2j_2 s} \leq (1+j)^{2(j_1+j_2)s}= (1+j)^{2j s}
  \end{align*}
  and (very) roughly estimate all other powers involved in terms of
  \begin{align*}
    D=100\max(C,D_1,D_2)^{2}.
  \end{align*}
\end{proof}

\begin{proof}[Proof of Theorem \ref{thm:quantitative_stability}]
  In the following we retrace and improve the proof in \cite{Zill5} and
  \cite[Section 4]{Zill3} in order to obtain a quantitative control of the
  constants in the stability estimate.
  
  We iteratively construct a family of Lyapunov functionals. That is, we claim
  that for all $j \in \N$ there exist non-increasing energies $E_j(t)$ such
  that
  \begin{align}
    \label{eq:1}
    C \|\omega(t)\|_{\dot{H}^{j}}^2&\leq E_{j}(t)\leq 2 \|\omega(t)\|_{\dot H^{j}}^2 + C^{j}\sum_{j_1+j_2=j, j_2\neq j} \|(f,g)\|_{j} \|\omega(t)\|_{H^{j_2}}^2.
  \end{align}
  The statement of the theorem then immediately follows by estimating
  \begin{align*}
    C\|\omega(t)\|_{\dot{H}^{j}}^2 \leq E_j(t)\leq E_j(0)\leq 2 \|\omega_0\|_{H^{j}}^2 + C^{j}\sum_{j_1+j_2<j} \|(f,g)\|_{W^{j_1,\infty}} \|\omega_0\|_{H^{j_2}}^2 .
  \end{align*}
  Here, it turns out that the main challenge lies in constructing the first
  energy functional $E_0(t)$ and establishing sufficiently good control of $\dt
  E_0(t)$. Energies $E_j(t)$ with
  larger $j$ can then be constructed inductively.\\

  \underline{The case $j=0$:}
  
  In order to introduce ideas, let us recall the damping mechanism, known as the
  Orr mechanism, in case of Couette flow $U(y)=y$. In this case
  $\omega(t,x,y)=\omega_0(x-ty,y)$ and as a result
  \begin{align*}
    \p_x \Delta^{-1} \omega \leadsto \frac{ik}{k^2+\eta^2} \tilde{\omega}_0(k,\eta+kt),
  \end{align*}
  where $\tilde{\omega}$ denotes the Fourier transform.
  Changing to coordinates moving with the flow $(x+ty,y)$ and thus $(k,\eta-kt)$
  we thus obtain the multiplier
  \begin{align*}
    \frac{k}{k^2+(\eta-kt)^2} \tilde{\omega}_0(k,\eta).
  \end{align*}
  This multiplier illustrates the main properties of the damping mechanism:
  \begin{itemize}
  \item As $t\rightarrow \infty$ the multiplier $\frac{k}{k^2+(\eta-kt)^2}$
    tends to zero (at an algebraic rate) and the velocity hence asymptotically
    converges in $L^2$.
  \item In contrast $\omega(t,x,y)=\omega_0(x-ty,y)$ does not converge strongly
    in $L^2$ but only weakly.
  \item While $\frac{k}{k^2+(\eta-kt)^2}$ decays after the time
    $t_c=\frac{\eta}{k}$ before that time the multiplier is actually increasing.
    Furthermore the operator norm on $L^2$,
    \[\sup_{k,\eta}
      \frac{k}{k^2+(\eta-kt)^2}=\sup_{k,\eta} \frac{k}{k^2+\eta^2}\] does not
    improve in time.
  \item However, if we can fix $k$ and $\eta$, then $\frac{k}{k^2+(\eta-kt)^2}$ is
    integrable in time.
  \end{itemize}
  Building in particular on the last property we thus aim to construct an energy
  $E_0(t)$ such that $-\dt E_0(t)\geq 0$ controls
  \begin{align*}
    |\langle \omega, U''(y)\p_x \Delta^{-1}\omega \rangle|.
  \end{align*}\\
  
  As the coefficient functions in problem \eqref{eq:linEuler1} do not depend on
  $x$ the problem decouples with respect to Fourier modes $k$ in $x$ and we may
  without loss of generality restrict to considering $\omega$ being restricted
  to a single arbitrary but fixed mode $k$. If $g\geq C>0$ we define the
  $H^{1}_t$ energy by
  \begin{align*}
    \|u\|_{H^1_{t}}^2= \int k^2 |u|^2 + C^2 |(\p_z-ikt)u|^2
  \end{align*}
  and we define the dual $H^{-1}_t$ energy of a function $u \in L^2$ in terms of a Fourier
  weight
  \begin{align*}
    \|u\|_{H^{-1}_t}^2 := \sum_k\int |\tilde{u}|^2 \frac{1}{k^2+C^2(\eta-kt)^2} d\eta.
  \end{align*}
  In particular, we note that this multiplier is integrable in time and hence
  define the Fourier multiplier $A(t)$ by
  \begin{align*}
    \mathcal{F}(Au) = \exp(c \arctan(C(\eta-kt)) \tilde{u}(k,\eta),
  \end{align*}
  where $0<c<1$ is a constant. This multiplier is non-increasing and it holds
  that for any $u$ not depending on time
  \begin{align*}
    \|u\|_{H^{-1}_t}^2 \exp(-c\pi) \leq -\dt \langle u, A u \rangle \leq \exp(c\pi)\|u\|_{H^{-1}_t}^2.
  \end{align*}
  We then make the ansatz
  \begin{align*}
    E_0(t):= \langle \omega(t), A(t) \omega(t) \rangle.
  \end{align*}
  As $\exp(c \arctan(C(\eta-kt))$ is bounded above and below it holds that
  \begin{align*}
    \exp(-c\pi)\|\omega(t)\|_{L^2}^2 \leq E_0(t) \leq \exp(c\pi) \|\omega(t)\|_{L^2}^2,
  \end{align*}
  so \eqref{eq:1} holds. It remains to verify that $E_0(t)$ is non-increasing.
  We estimate
  \begin{align*}
    \frac{d}{dt} E_0(t) = \langle \omega(t), (\dt A) \omega(t) \rangle + 2 \langle f\p_x L_t \omega(t), A(t) \omega(t) \rangle \\
    \leq -\exp(-c\pi) \|\omega(t)\|_{H^{-1}_t}^2 + 2 \langle f\p_x L_t \omega(t), A(t) \omega(t) \rangle.
  \end{align*}
  Using duality we then estimate
  \begin{align*}
    |\langle f\p_x L_t \omega(t), A(t) \omega(t) \rangle| \leq \|A(t) \omega(t)\|_{H^{-1}_t} \|f\p_x L_t \omega(t)\|_{H^1_t} \\
    \leq \exp(c\pi) \|\omega(t)\|_{H^{-1}_t} \|f\|_{W^{1,\infty}} \|\p_x L_t \omega(t)\|_{H^1_t}.
  \end{align*}
  Lastly, recall that $L_t \omega$ solves
  \begin{align*}
    (-k^2+(g(\p_y-ikt))^2) L_t \omega = \omega.
  \end{align*}
  Testing this equation with $-\frac{1}{g} L_t \omega$ and using that $g$ is
  bounded below we thus obtain that
  \begin{align*}
    \|L_t\omega\|_{H^{1}_t}^2 \leq \langle -\frac{1}{g L_t \omega}, \omega  \rangle \leq \|\omega\|_{H^{-1}_t} \|\frac{1}{g}\|_{C^1}\|L_t\omega\|_{H^{1}_t}. 
  \end{align*}
  and thus
  \begin{align*}
    \|\p_x L_t \omega(t)\|_{H^1_t} \leq \|\frac{1}{g}\|_{C^1} \|\omega(t)\|_{H^{-1}_t}.
  \end{align*}
  Thus, if $\|f\|_{W^{1,\infty}}$ is sufficiently small, it holds that
  \begin{align*}
    \frac{d}{dt} E_0(t)  + C \|\omega(t)\|_{H^{-1}_t}^2 \leq 0
  \end{align*}
  and in particular if follows that $E_0(t)$ non-increasing.\\

  \underline{The induction step:}  \\
  
  Based on the above estimate we claim that in addition to \eqref{eq:1} it holds
  that
  \begin{align}
    \label{eq:2}
    \begin{split}
      \frac{d}{dt}E_{j}(t)&\leq - C\|\p_y^{j}\omega(t)\|_{H^{-1}_t}^2 - C
      \|(f,g)\|_{1}^2\|\p_y^{j-1}\omega(t)\|_{H^{-1}_t}^2 \\ & \quad - \dots -C
      \|(f,g)\|_{j}^2 \|\omega(t)\|_{H^{-1}_t}^2,
    \end{split}
  \end{align}
  which we have just established for $j=0$. We then make the ansatz
  \begin{align}
    \label{eq:3}
    \begin{split}
      E_0(t)&= \langle \omega, A \omega \rangle, \\
      E_{j+1}(t)&:= 2 \langle \p_y^{j+1}\omega, A \p_y^{j+1}\omega \rangle + 4C
      \sum_{j_1+j_2=j} \|(f,g)\|_{j_1}^2 E_{j_2}(t).
    \end{split}
  \end{align}
  In particular, by construction this satisfies \eqref{eq:1} for every $j$. It
  remains to be shown that $E_{j+1}(t)$ satisfies \eqref{eq:2} and hence is
  non-increasing.

  Thus, consider the $\p_y^{j+1}$ derivative of the linearized Euler equations:
  \begin{align*}
    \dt \p_y^{j+1} \omega + f \p_x L_t \p_y^{j+1}\omega = [f\p_x L_t,\p_y^{j+1}]\omega.
  \end{align*}
  By the construction of $A(t)$ we obtain that
  \begin{align*}
    \frac{d}{dt} 2 \langle \p_y^{j+1}\omega, A \p_y^{j+1}\omega \rangle &= 2 \langle \p_y^{j+1}\omega, \dot A \p_y^{j+1}\omega \rangle \\
                                                                        & \quad +4 \langle -f \p_x L_t \p_y^{j+1} \omega , A \p_y^{j+1}\omega \rangle + 4 \langle [f \p_x L_t, \p_y^{j+1}] \omega , A \p_y^{j+1}\omega \rangle \\
                                                                        &\leq - C\|\p_y^{j}\omega(t)\|_{H^{-1}_t}^2 +  4 \langle [f \p_x L_t, \p_y^{j+1}] \omega , A \p_y^{j+1}\omega \rangle.
  \end{align*}
  Using \eqref{eq:2} up to $j$ and our ansatz \eqref{eq:3} it thus suffices to
  show that
  \begin{align*}
    \langle [f \p_x L_t, \p_y^{j+1}] \omega , A \p_y^{j+1}\omega \rangle &\leq C \|\p_y^{j+1}\omega\|_{H^{-1}_t}\sum_{j_1+j_2=j} \|(f,g)\|_{j_1}|\p_y^{j_2}\omega\|_{H^{-1}_t},
  \end{align*}
  at which point we can then conclude our estimate by using Young's inequality.

  Indeed, by duality we may control
  \begin{align*}
    \langle [f \p_x L_t, \p_y^{j+1}] \omega , A \p_y^{j+1}\omega \rangle \leq \|A \p_y^{j+1}\|_{H^{-1}_t} \|[f \p_x L_t, \p_y^{j+1}] \omega \|_{H^{1}_t},
  \end{align*}
  and by construction of $A$
  \begin{align*}
    \|A \p_y^{j+1}\omega\|_{H^{-1}_t} \leq c \|\p_y^{j+1}\omega \|_{H^{-1}_t}.
  \end{align*}
  We may thus focus on computing and estimating the commutator. Here, the $j+1$
  derivatives may fall either on $f$ or on $\p_xL_t$ and we can estimate
  \begin{align*}
    \|(\p_y^{j_1}f) \p_y^{j_2} \p_x L^{t} \omega \|_{H^{1}_t} \leq \|\p_y^{j_1}f\|_{L^\infty} \|\p_y^{j_2} \p_x L^{t} \omega \|_{H^{1}_t} +  \|\p_y^{j_1+1}f\|_{L^\infty}\|\p_y^{j_2} \p_x L_t \omega \|_{L^2}.
  \end{align*}
  Therefore using the structure of $\|(f,g)\|_{j}$ we may further reduce to
  studying $ \|\p_y^{j_2} \p_x L_t \omega \|_{H^{1}_t}$. Using the definition of
  $L_t$ and the fact that in this setting of an infinite channel we need not
  worry about boundary conditions, we observe that $\p_y^{j_2} L_t \omega$ is
  the unique solution of
  \begin{align*}
    (-k^2+(g(\p_z-ikt))^2) \p_y^{j_2} L_t \omega = \p_y^{j_2} \omega + [(-k^2+(g(\p_z-ikt))^2), \p_y^{j_2}] L_t \omega .
  \end{align*}
  Again using the ellipticity of this problem (see also \cite{Zill3}) it thus
  follows that
  \begin{align*}
    \|\p_y^{j_2} L_t \omega\|_{H^1_t}^2 \leq \|\p_y^{j_2} \omega\|_{H^{-1}_t}^2 + \|[(-k^2+(g(\p_z-ikt))^2), \p_y^{j_2}] L_t \omega\|_{H^{-1}_t}^2.
  \end{align*}
  Inductively repeating this argument for the commutator on the right-hand-side,
  we may estimate
  \begin{align*}
    \|\p_y^{j_2} L_t \omega\|_{H^1_t}^2 \leq \|\p_y^{j_2} \omega\|_{H^{-1}_t}^2 + \sum_{j_3+j_4=j_2} \|g\|_{j_3}^2 \|\p_y^{j_4} \omega\|_{H^{-1}_t}^2,
  \end{align*}
  which concludes the proof.
\end{proof}

We thus observe that the main challenge of the proof in this infinite channel
setting is given by establishing the result at the level of $L^2$. Higher
Sobolev space estimates may then be obtained inductively by using commutator
estimates.  Furthermore Gevrey stability estimates then correspond to good quantitative control
of the constants in these estimates.

As one of the main results of \cite{Zill5}
we showed that in the setting of a finite channel additional corrections due to
boundary effects have to be taken into account and are generally not negligible,
leading to asymptotic instability. However, in the special case where $\omega$
and $f$ vanish to sufficiently high order on the boundary this instability does
not manifest (up to this order). In particular, as we show in the following
section, if the perturbation and $f$ are compactly supported as in the setting
considered by Jia \cite{jia2019linear} the above proof essentially extends to
the setting with boundary with some minor modifications.

\section{The Finite Channel Case}
\label{sec:finite}

In this section we consider the setting of a finite channel $\T_L \times [0,1]$
\begin{align}
  \label{eq:channel}
  \begin{split}
    \dt \omega + f \p_x L_t \omega &=0, \\
    (\p_x^2+(g(\p_y-t\p_x))^2)L_t \omega &=0, \\
    L_t \omega |_{y=0,1}&=0.
  \end{split}
\end{align}
Here, in addition to technical challenges such finding suitable basis
representations to replace the Fourier transform a major obstacle is given by
the boundary conditions imposed on the stream function. Indeed, as one of the
main results of \cite{Zill5} we showed that these corrections are generically
not integrable in time and results in blow-up in $H^s, s>\frac{3}{2}$ if $f
\omega_0$ does not vanish on the boundary.

That is, while $L_t \omega(t)$ is prescribed to satisfy impermeable wall
conditions (which equals zero Dirichlet conditions after removing the $x$
average), $\p_y^{j} L_t \omega$ is not given by the unique solution to
\begin{align*}
  (-k^2+(g(\p_y-ikt))^2) \psi &= \p_y^{j} \omega + [(-k^2+(g(\p_y-ikt))^2), \p_y^{j}] L_t \omega, \\
  \psi|_{y=0,1}&=0,
\end{align*}
since generically $\p_y^{j} L_t \omega|_{y=0,1}\neq 0$.

Hence, in order to compute $\p_y^{j}L_t \omega$ we need to include additional
boundary corrections:
\begin{align}
  \label{eq:5}
  \begin{split}
    \p_y^{j}L_t \omega &= L_t \p_y^{j} \omega + L_t [(g(\p_y-ikt))^2, \p_y]L_t \omega \\
    & \quad + (\p_y^{j}L_{t} \omega)(0) e^{ikty}u_1 + (\p_y^{j}L_{t} \omega)(1)
    e^{ikt(y-1)}u_2,
  \end{split}
\end{align}
where $e^{ikty} u_1, e^{ikt(y-1)}u_2$ are homogeneous solutions of the stream
function problem (see Proposition \ref{prop:Dirichlet}).

However, if $\omega_0$ and $f$ happen to be supported in $I \subset (0,1)$
\begin{align}
  \label{eq:compact}
  \text{supp}(\omega_0) \subset \T \times I, \text{supp}(f) \subset I,   
\end{align}
this instability can be avoided as shown in \cite{jia2019linear} and
\cite{ionescu2018inviscid} for the linear and nonlinear Euler equations,
respectively.

In the following we show that under this support assumption linear stability in
arbitrary Sobolev spaces and Gevrey regularity follow as an extension of the
$H^{1}$ stability results in \cite{Zill5} and thus provide a new short proof of
the former result (for a different class of shear flows). Furthermore, we also
consider the setting where $f$ and $\omega_0$ vanish of (at least) a finite
order $N$:
\begin{align}
  \tag{$V_N$}
  \label{eq:vanishing}
  \forall j \leq N, \p_y^j\omega_0(0)=\p_y^j\omega_0(1)=\p_y^jf(0)=\p_y^jf(1)=0.
\end{align}

As a first observation we note that \eqref{eq:compact} and \eqref{eq:vanishing}
are preserved under the evolution.
\begin{lem}
  \label{lem:support}
  Let $I \subset (0,1)$ be a closed subinterval and suppose that $f \in C^1$,
  $\omega_0 \in L^2$ satisfy \eqref{eq:compact}. Then for any $t\geq 0$ the
  solution $\omega$ of \eqref{eq:channel} satisfies
  \begin{align*}
    \text{supp}(\omega(t)-\omega_0) \subset \T \times I.
  \end{align*}

  Similarly, if $f \in C^{N+1}, \omega_0\in H^{N+1}$ satisfy
  \eqref{eq:vanishing}, then $\omega(t)-\omega_0$ vanishes to order at least $N$
  on the boundary.
\end{lem}

\begin{proof}[Proof of Lemma \ref{lem:support}]
  Suppose that $f$ is supported in the interval $I$. Then
  \begin{align*}
    \dt \omega = -f \p_x L_t \omega
  \end{align*}
  vanishes if $f$ vanishes and hence
  \begin{align*}
    \omega(t)-\omega_0= \int_0^t \dt \omega
  \end{align*}
  is supported in $\T_L\times I$.

  Concerning the finite order of vanishing we note that $\dt \omega= -f \p_x L_t
  \omega$ vanishes of order at least $N$, since $L_t \omega$ vanishes on the
  boundary (though it might vanish arbitrarily slowly) and $f$ by assumption
  vanishes to order at least $N$. Integrating over the compact time interval
  $[0,t]$ it thus follows that $\omega(t)-\omega_0$ vanishes of order at least
  $N$.
\end{proof}

This allows us to establish improved estimates on $\p_y^{j}L_t \omega|_{y=0,1}$.
\begin{prop}
  \label{prop:Dirichlet}
  Suppose that for some $N \in \N$, $f$ and $\omega_0$ satisfy
  \eqref{eq:vanishing}. For any $j \leq N$ it holds that
  \begin{align}
    \label{eq:4}
    |\p_y^j L_t \omega|_{y=0,1}| \leq (1+\|g\|_{j})\langle t \rangle^{j-1} \p_y L_t \omega|_{y=0,1}.
  \end{align}

  Furthermore, for any $\delta>0$ we may estimate
  \begin{align}
    \label{eq:6}
    \p_y L_t \omega|_{y=0,1} \leq C_{\delta} \langle t \rangle^{-j} \sum_{j_1+j_2=j} \|g\|_{j_1}  \sqrt{\sum_{\eta} \frac{1}{\langle \eta-kt \rangle^{1-\delta}}|\tilde{\p_y^{j_2}\omega}(\eta)|^2}.
  \end{align}
\end{prop}

\begin{proof}[Proof of Proposition \ref{prop:Dirichlet}]
  In the case $j=1$ the estimate \eqref{eq:4} is an equality and we recall that
  $L_t\omega|_{y=0,1}=0$ by definition of $L_t$. For $j\geq 2$ we may reduce to
  the above estimates by noting that
  \begin{align*}
    (g(\p_y-ikt))^{j} L_t \omega = k^2 (g(\p_y-ikt))^{j-2}L_t \omega + (g(\p_y-ikt))^{j-2}\omega. 
  \end{align*}
  If $j-2\leq N$ the last term vanishes and we hence obtain a recursion formula
  \begin{align*}
    (g(\p_y-ikt))^{j} L_t \omega = k^2 (g(\p_y-ikt))^{j-2}L_t \omega,
  \end{align*}
  which we can solve for $\p_y^{j}L_t \omega$ since $g$ is bounded above and
  below. The estimate \eqref{eq:4} thus
  immediately follows by induction.\\

  It thus remains to study $\p_yL_t \omega|_{y=0,1}$. As shown in \cite[Lemma
  3]{Zill3} we may compute this Neumann data as integrals against $\omega$. That
  is, if $u_{0}, u_1$ are solutions of the adjoint problem
  \begin{align}
    \label{eq:7}
    (-k^2+((\p_y-ikt)g)^2)u =0
  \end{align}
  with boundary conditions
  \begin{align*}
    \begin{pmatrix}
      u_0(0) & u_1(0)\\
      u_0(1) & u_1(1)
    \end{pmatrix}
               =
               \begin{pmatrix}
                 -1 & 0 \\ 0 &1
               \end{pmatrix},
  \end{align*}
  then it follows that
  \begin{align}
    \label{eq:8}
    \begin{split}
      \int u \omega = \int (-k^2+(g(\p_y-ikt))^2)L_t \omega &= u g^2 (\p_y-ikt)L_t \omega |_{y=0,1}\\
      &=
      \begin{cases}
        g^2(0)\p_yL_t\omega (0) & \text{ if } u=u_0, \\
        g^2(1)\p_yL_t\omega (1) & \text{ if } u=u_1,
      \end{cases}
    \end{split}
  \end{align}
  where we again used that $L_t\omega|_{y=0,1}$. We note that by the structure
  of \eqref{eq:7} it holds that $u_0(t,y)=e^{ikty}u_0(0,y)$ and
  $u_1(t,y)=e^{ikt(y-1)}u_1(0,y)$ and that
  \begin{align*}
    u_0(0,y)&= \frac{1}{g}\frac{\sinh(k(U^{-1}(y)-U^{-1}(1)))}{\sinh(k(U^{-1}(0)-U^{-1}(1)))}, \\
    u_1(0,y)&= \frac{1}{g} \frac{\sinh(k(U^{-1}(y)-U^{-1}(0)))}{\sinh(k(U^{-1}(0)-U^{-1}(1)))},
  \end{align*}
  can be explicitly computed and are smooth functions.

  As $\omega$ vanishes on the boundary we may integrate by parts $j$ times in
  \eqref{eq:8} and thus obtain that
  \begin{align*}
    \p_y L_{t} \omega(0)= \frac{k}{g^2(0)} \frac{1}{(-ikt)^j}\int e^{ikty} \p_y^{j}(u_0(0,y) \omega).
  \end{align*}
  Expanding $\p_y^{j}(u_1(y)\omega)$ by the product rule and $\p_y^{j_2}\omega$
  in terms of its Fourier series, the estimate then follows by noting that
  \begin{align*}
    \frac{1}{k^{j_1}} \left| \int e^{i(kt-\eta)y} \p_y^{j_1} u_1 \right| \leq \frac{1}{\langle \eta-kt \rangle}
  \end{align*}
  and that by the Cauchy-Schwarz inequality for any sequence $R \in l^2$ it
  holds that
  \begin{align*}
    \sum_\eta \frac{1}{\langle \eta-kt \rangle} |R_\eta| \leq \sqrt{\sum_\eta \frac{1}{\langle \eta-kt \rangle^{1+\delta}}}  \sqrt{\sum_\eta \frac{1}{\langle \eta-kt \rangle^{1-\delta}} |R_\eta|^2}.
  \end{align*}
  The result hence follows with $C_\delta=\sqrt{\sum_\eta \frac{1}{\langle
      \eta-kt \rangle^{1+\delta}}} <\infty$.
\end{proof}

With these preparations we are now ready to prove Theorem \ref{thm:finite}.

\begin{proof}[Proof of Theorem \ref{thm:finite}]
  As in Section \ref{sec:infinite} we iteratively construct a family of Lyapunov
  functionals. Our basic building block is given by the weight $A(t)$ from
  \cite[Lemma 5.3]{Zill3}
  \begin{align*}
    A(t): e^{i \eta y} \exp(\arctan(\frac{\eta}{k}-t)- \int_0^{t}\frac{1}{\langle \tau \rangle^{2\beta}}\frac{1}{(1+(\eta/k-\tau)^{2})^{2\gamma}} d\tau) e^{i \eta y},
  \end{align*}
  which clearly satisfies
  \begin{align}
    \label{eq:10}
    \langle u, \dot A(t) u \rangle \leq -C \sum_{\eta}|u(\eta)|^2 \left(\frac{1}{1+(\eta/k-t)^2} + \frac{1}{\langle t \rangle^{2\beta}}\frac{1}{(1+(\eta/k-t)^{2})^{2\gamma}}  \right).
  \end{align}
  Here the second term will be used to control contributions due to $\p_y^{j}L_t
  \omega|_{y=0,1}$.

  Following the same strategy as in the proof of Theorem
  \ref{thm:quantitative_stability} we make the ansatz
  \begin{align}
    \label{eq:9}
    \begin{split}
      E_0(t)&= \langle \omega, A \omega \rangle, \\
      E_{j+1}(t)&:= 2 \langle \p_y^{j+1}\omega, A \p_y^{j+1}\omega \rangle + 4C
      \sum_{j_1+j_2=j} \|(f,g)\|_{j_1}^2 E_{j_2}(t).
    \end{split}
  \end{align}
  and claim that for all $j$
  \begin{align}
    \label{eq:13}
    \begin{split}
      \frac{d}{dt} E_j(t) &\leq -C \sum_{\eta} (\frac{1}{1+(\eta/k-t)^2} + \frac{1}{\langle t \rangle^{2\beta} (1+(\eta/k-t)^2)^{\gamma}}) |\p_y^{j}\omega(\eta)|^2 \\
      &\quad - C \sum_{j_1+j_2=j-1} \|\p_y(f,g)\|_{j_1} \\
      & \quad \times\sum_{\eta} (\frac{1}{1+(\eta/k-t)^2} + \frac{1}{\langle t
        \rangle^{2\beta} (1+(\eta/k-t)^2)^{\gamma}}) |\p_y^{j_2}\omega(\eta)|^2
      \leq 0.
    \end{split}
  \end{align}
    
  By construction it then again holds that
  \begin{align*}
    C\|\omega(t)\|_{\dot{H}^{j}}^2 \leq E_j(t)\leq E_j(0)\leq 2 \|\omega_0\|_{H^{j}}^2 + C^{j}\sum_{j_1+j_2<j} \|(f,g)\|_{W^{j_1,\infty}} \|\omega_0\|_{H^{j_2}}^2, 
  \end{align*}
  which implies the result.

  \underline{The case $j=0$:} In the following we recall the construction of
  $E_0(t)$ and $E_1(t)$ from \cite{Zill3} and subsequently extend our proof to
  the case of general $j$. We claim that there exists $c>0$ such
  \begin{align}
    \label{eq:11}
    \langle A(t)\omega, ikf L_t \omega \rangle \leq ck\|f\|  \sum_{\eta}|\tilde{\omega}(\eta)|^2 \frac{1}{1+(\eta/k-t)^2}. 
  \end{align}
  Then if $\|f\| k c< C/2$ with $C$ as in \eqref{eq:10} it immediately follows
  that $E_0(t)$ is non-increasing and furthermore
  \begin{align*}
    \frac{d}{dt}E_0(t)\leq - C/2 \sum_{\eta}|\tilde{\omega}(\eta)|^2 \frac{1}{1+(\eta/k-t)^2}.
  \end{align*}
  Indeed, suppose that $g>c>0$ and for any $u \in L^2$ and any $t>0$ define
  $\Lambda_t[u]$ to be the unique solution of
  \begin{align*}
    (-k^2+c^2(\p_y-ikt)^2) \Lambda_t[u]&=u, \\
    \Lambda_t[u]|_{y=0,1}&=0.
  \end{align*}
  That is, we replaced $g$ by a constant. We then define
  \begin{align*}
    \|u\|_{H^1_t}^2:= k^2 \|u\|_{L^2}^2 + c^2 \|(\p_y-ikt)u\|_{L^2}^2
  \end{align*}
  and
  \begin{align*}
    \|u\|_{H^{-1}_t}^2 := - \langle \Lambda_t[u], u \rangle = \|\Lambda_t[u]\|_{H^{1}_t}^2.
  \end{align*}
  These are by construction dual norms, so we may estimate
  \begin{align*}
    \langle A(t)\omega, ikf L_t \omega \rangle \leq \|A(t)\omega\|_{H^{-1}_t} k \|f\|_{C^1} \|L_t \omega\|_{H^{1}_t}.
  \end{align*}
  Since $L_t \omega$ is defined in terms of an elliptic operator we may further
  estimate
  \begin{align*}
    \|L_t \omega\|_{H^{1}_t}^2 \leq -\langle \omega, L_t \omega \rangle \leq \|\omega\|_{H^{-1}_t} \|L_t \omega\|_{H^{1}_t}^2.
  \end{align*}
  Our estimate thus follows if we can show that $\|u\|_{H^{-1}_t}^2$ is
  controlled by a Fourier multiplier as in \eqref{eq:11}. In the whole-space
  setting of Section \ref{sec:infinite} such a result is trivial since
  $(-k^2+c^2(\p_y-ikt)^2)$ is given by a Fourier multiplier and thus $H^{-1}_t$
  is as well. In the present setting with boundary this poses some technical challenges and one can show by explicit computation of
  $\Lambda_t[e^{i\eta y}]$ (see \cite{Zill3} Lemma 5.2) that indeed
  \begin{align*}
    \|u\|_{H^{-1}_t}^2 \leq c \sum_{\eta}|u(\eta)|^2 \frac{1}{1+(\eta/k-t)^2}.
  \end{align*}
  This concludes the proof for the case $j=0$. As shown in \cite{Zill5} a
  similar result is also valid in fractional Sobolev spaces.
  \\

  \underline{The induction step:} Having established the base case of
  \eqref{eq:13} we now consider the induction step.
    
  We may write the $\p_y^{j+1}$ derivative of \eqref{eq:channel} as
  \begin{align}
    \label{eq:12}
    \dt \p_y^{j} \omega + \sum_{j_1+j_2=j}{j \choose j_1} (\p_y^{j_1}f) \p_y^{j_2}L_t \omega =0.
  \end{align}
  Furthermore, we may split
  \begin{align*}
    \p_y^{j_2}L_t \omega &= (\p_y^{j_2}L_t \omega- \p_y^{j_2}L_t \omega(0)u_0 - \p_y^{j_2}L_t \omega(1)u_1) + \p_y^{j_2}L_t \omega(0)u_0 + \p_y^{j_2}L_t \omega(1)u_1\\
                         &=:\psi_j + \p_y^{j_2}L_t \omega(0)u_0 + \p_y^{j_2}L_t \omega(1)u_1.
  \end{align*} 
  As the $H^{1}_t$ norms of $u_0(t,y)=e^{ikty}u_0(0,y)$ and
  $u_1(t,y)=e^{ikt(y-1)}u_1(0,y)$ are independent of $t$ it follows that
  \begin{align}
    \label{eq:14}
    \|\p_y^{j}L_t \omega\|_{H^{1}_t} \leq  \|\psi_j\|_{H^{1}_t} + C|\p_y^{j}L_t \omega| (0) + C|\p_y^{j}L_t \omega| (1).
  \end{align}
  As $\psi_j$ satisfies zero Dirichlet boundary conditions, we may test
  \begin{align*}
    (-k^2+(g(\p_y-ikt))^2) \psi_j =\p_y^{j}\omega(t) + [(-k^2+(g(\p_y-ikt))^2), \p_y^j]L_t \omega
  \end{align*}
  with $-\psi_j$ and integrate by parts to obtain that
  \begin{align}
    \label{eq:15}
    \|\psi_j\|_{H^{1}_t} \leq \|\p_y^{j}\omega(t)\|_{H^{-1}_t} + \sum_{j_1+j_2=j, j_2<j} \|g\|_{j_1} \|\p_y^{j_2}L_t\omega\|_{H^{1}_t}. 
  \end{align}
  Combining \eqref{eq:14} and \eqref{eq:15} it follows that
  \begin{align}
    \|\psi_j\|_{H^{-1}_t} \leq \|\p_y^{j}\omega\|_{H^{-1}_t} + \sum_{j_1+j_2=j, j_2<j} \|g\|_{j_1} (\|\p_y^{j_2}\omega\|_{H^{-1}_t} +|\p_y^{j_2}L_t \omega|_{y=0,1}|).
  \end{align}
  We may now further invoke Proposition \ref{prop:Dirichlet} to estimate
  \begin{align}
    |\p_y^{j_2}L_t \omega|_{y=0,1}| \leq C_{\delta} \langle t \rangle^{-1} \sum_{j_3+j_4=j_2} \|g\|_{j_3} \sqrt{\sum_{\eta} \frac{1}{\langle \eta-kt \rangle^{1-\delta}} |\p_y^{j_4}\omega(\eta)|^2}.
  \end{align}
  As a final tool, we note that for $u \in \{u_0,u_1\}$ due to the oscillatory
  structure it holds that
  \begin{align}
    |\langle A \p_y^{j+1} \omega, u \rangle | \leq C_\delta \sqrt{\sum_{\eta}\frac{1}{\langle \eta/k-t \rangle^{1-\delta}} |\p_y^{j+1}\omega(\eta)|^2}
  \end{align}

  With all these estimates at hand, we may integrate \eqref{eq:12} against $A
  \p_y^{j+1}\omega$ and control
  \begin{align}
    |\langle \dt \p_y^{j+1}\omega, A \p_y^{j+1}\omega \rangle|  &\leq  \|f\|_{W^{1,\infty}} \|\p_y^{j+1}\omega\|_{H^{-1}_t}^2 \\
                                                                & \quad + \|f\|_{W^{1,\infty}} C_{\delta} \langle t \rangle^{-1} \sum_{\eta}\frac{1}{\langle \eta/k-t \rangle^{1-\delta}} |\p_y^{j+1}\omega(\eta)|^2 \\
                                                                & \quad + \sum_{j_1+j_2=j+1, j_2<j+1}\|g\|_{j_1} (\|\p_y^{j_2}\omega\|_{H^{-1}_t}+  |\p_y^{j_2}\omega(\eta)|^2).
  \end{align}
  The first two terms are exactly such that \eqref{eq:10} shows that they can be
  absorbed into
  \begin{align*}
    \langle \p_y^{j+1}\omega, \dot A \p_y^{j+1}\omega \rangle,
  \end{align*}
  provided $f$ satisfies the smallness assumption. The remaining terms are all
  of lower order and using \eqref{eq:13} can be absorbed into
  \begin{align*}
    \frac{d}{dt} 4C \sum_{j_1+j_2=j} \|(f,g)\|_{j_1}^2 E_{j_2}(t)
  \end{align*}
  by the induction assumption. Thus, indeed $E_{j+1}(t)$ satisfies
  \eqref{eq:13}, which concludes the proof.
\end{proof}

\section{Discussion}
\label{sec:disc}
In this article we show that stability in Gevrey regularity corresponds to a
quantitative control of stability in Sobolev spaces.
Furthermore, this quantitative essentially reduces to establishing good
estimates in $L^2$ and $H^1$ as in \cite{Zill5}, \cite{Zill3}, which then extend
to arbitrary Sobolev regularity.
We thus provide a new perspective on and a very short alternative proof of the
results of Jia \cite{jia2019linear}.

Furthermore, we consider the settings of both infinite and finite channels and
the milder constraint of a high finite order of vanishing instead of requiring
compact support.
In particular, vanishing of infinite order is shown to be sufficient to
establish stability in Gevrey regularity.
On the other hand, our perturbative construction of the energy functional $E_0$
imposes a smallness condition instead of a sharper non-resonance condition.

A natural question in view of the existing instability results in $H^{3/2+}$
(\cite{Zill5}) for perturbations not vanishing on the boundary and the stability
results of Theorem \ref{thm:finite2} is to which extent the condition
\eqref{eq:vanishing} is necessary for (asymptotic) stability to hold, both in
the linear and nonlinear setting.
Here, the analysis in \cite{Zill5} suggests to consider boundary corrections and
search for $j$ such that
\begin{align*}
  \dt \p_y^{j} \omega|_{y=0,1}=  -\p_y^{j}(f\p_xL_t \omega)
\end{align*}
is not integrable in time. This would then imply an instability in
$W^{j,\infty}$ and by the Sobolev embedding also an instability in higher Sobolev
regularity.
In this view a sharper formulation of \eqref{eq:vanishing} hence might be to
impose an order of vanishing $N$ on the product $f \omega_0$ instead. However,
condition \eqref{eq:vanishing} allows for a simple formulation of Proposition \ref{prop:Dirichlet}.

\subsection*{Acknowledgments}
Christian Zillinger's research is supported by the ERCEA under the grant 014
669689-HADE and also by the Basque Government through the BERC 2014-2017
program and by Spanish Ministry of Economy and Competitiveness MINECO: BCAM Severo Ochoa excellence accreditation SEV-2013-0323.
\bibliographystyle{alpha} \bibliography{citations2}
\end{document}